\documentclass{article}

\usepackage{amssymb,amsmath,amsthm}
\usepackage{mathrsfs,color}
\usepackage
{hyperref}
\usepackage{epsfig}\usepackage{graphicx}
 \newtheorem{thm}{Theorem}[section]
 \newtheorem{cor}[thm]{Corollary}
 \newtheorem{lem}[thm]{Lemma}
 \newtheorem{prop}[thm]{Proposition}
 \newtheorem{defn}[thm]{Definition}
 \newtheorem{rem}[thm]{Remark}
 \newtheorem{ex}{Example}
 
\newcommand{\Cinf}{\ensuremath{\mathcal{C}^\infty}}

\newcommand{\D}{\ensuremath{{\mathcal D}}}

\newcommand{\E}{\ensuremath{{\mathcal E}}}

\newcommand{\mb}[1]{\ensuremath{\mathbb{#1}}}
\newcommand{\N}{\mb{N}}

\newcommand{\R}{\mb{R}}
\newcommand{\C}{\mb{C}}

\newcommand{\G}{\ensuremath{{\mathcal G}}}

\newcommand{\EM}{\ensuremath{{\mathcal E}_{\mathrm{M}}}}

\newcommand{\NN}{\ensuremath{{\mathcal N}}}

\newcommand{\lara}[1]{\langle #1 \rangle}
\newcommand{\WF}{\mathrm{WF}}
\newcommand{\singsupp}{\mathrm{sing supp}}
\newcommand{\Char}{\ensuremath{\text{Char}}}

\newcommand{\grad}{\ensuremath{\mbox{\rm grad}\,}}
\renewcommand{\div}{\ensuremath{\mbox{\rm div}\,}}

\newfont{\bl}{msbm10 scaled \magstep2}

\newcommand{\beq}{\begin{equation}}
\newcommand{\eeq}{\end{equation}}

\newcommand{\isom}{\cong}
\newcommand{\col}{\colon}

\newcommand{\notmid}{\mid\kern-0.5em\not\kern0.5em}

\newcommand{\be}{\beta}
\newcommand{\ga}{\gamma}
\newcommand{\Ga}{\Gamma}
\newcommand{\de}{\delta}
\newcommand{\eps}{\varepsilon}

\newcommand{\vphi}{\varphi}

\newcommand{\la}{\lambda}

\newcommand{\Om}{\Omega}

\newcommand{\supp}{\mathop{\mathrm{supp}}}

\newcommand{\comp}{\subset\subset}
\newcommand{\GaG}{\Gamma_\G}

\newcommand\bsig{\ensuremath{{\sigma}}}
\newcommand\bsighat{\ensuremath{\widehat{{{\sigma}}}}}
\newcommand\bsigepshat{\ensuremath{\widehat{{{\sigma}}}_{\eps}}}

\newcommand\bxi{\ensuremath{{{\xi}}}}
\newcommand\bxihat{\ensuremath{\widehat{{\xi}}}}
\newcommand\bxieps{\ensuremath{{{\xi}}_{\eps}}}
\newcommand\bxiepshat{\ensuremath{\widehat{{{\xi}}}_{\eps}}}

\newcommand{\SobSDt}[2]{^{\nabla}\|#1\|_{S_\tau,\,\eps}^{#2}}
\newcommand{\SobSDz}[2]{^{\nabla}\|#1\|_{S_\zeta,\,\eps}^{#2}}

\newcommand{\SobODt}[2]{^{\nabla}\|#1\|_{\Omega_\tau,\,\eps}^{#2}}

\title{Wave equations on non-smooth space-times}
\author{G\"unther H\"ormann \footnote{University of Vienna, Faculty of Mathematics, guenther.hoermann@univie.ac.at}
\\
Michael Kunzinger\footnote{University of Vienna, Faculty of Mathematics, michael.kunzinger@univie.ac.at} \\
Roland Steinbauer\footnote{University of Vienna, Faculty of Mathematics, roland.steinbauer@univie.ac.at}}

\begin{document}
\date{}
\maketitle

\begin{abstract}
We consider wave equations on Lorentzian manifolds in case of low regularity.
We first extend the classical solution theory to prove global unique solvability 
of the Cauchy problem for distributional data and right hand side on smooth globally
hyperbolic space-times. Then we turn to the case where the metric is non-smooth
and present a local as well as a global existence and uniqueness result for a large class of 
Lorentzian manifolds with a weakly singular, locally bounded  metric in Colombeau's 
algebra of generalized functions.
  

\end{abstract}


\section{Introduction}

In this note we are concerned with wave equations on Lorentzian manifolds, with a
particular
interest in the Cauchy problem in non-smooth situations. To achieve a
self-contained presentation 
we first give a brief account on the classical solution theory in the smooth
case, in particular
discussing local well-posedness as well as global well-posedness on globally
hyperbolic manifolds.
We first extend this theory to the case of distributional data and right hand
sides. Then we turn to
the case where the metric is non-smooth, that is we deal with normally
hyperbolic operators
with coefficients of low regularity. Actually the regularity class of the metric
which we have 
in mind is below ${\mathcal C}^{1,1}$ (i.e., the first derivative locally
Lipschitz)---the largest 
class where standard differential geometric results such as existence and
uniqueness of geodesics 
remain valid, and also below the Geroch-Traschen class of metrics---the largest
class that
allows for a consistent distributional treatment (\cite{GT:87,LeFM:07}). It is
evident that
no consistent distributional solution concept is available for these equations. 
Therefore we consider a large class of weakly singular, locally bounded
metrics in the setting of generalized global analysis and nonlinear
distributional geometry
(\cite[Sec.\ 3.2]{GKOS:01}) based on Colombeau algebras of generalized functions
(\cite{Colombeau:84,Colombeau:85}). 
In particular, we present a local existence and uniqueness result in the spirit
of \cite{GMS:09}
and extend it to a global result on space-times with a generalized metric that
allows for a suitable globally
hyperbolic metric splitting. 

This line of research has drawn some motivation from general relativity:
In \cite{C:96} Chris Clarke suggested to replace the standard geometric
definition of singularities 
by viewing them as obstructions to the well-posedness of the initial value
problem for a scalar field. 
However, for many relevant examples as e.g. impulsive gravitational waves,
cosmic strings,
and shell crossing singularities which have a metric of low regularity the
Cauchy problem cannot consistently be 
formulated in distribution theory and one has to use a more sophisticated
solution concept. In case of shell crossing
singularities Clarke himself used a cleverly designed weak solution concept to
argue for local solvability of the wave equation 
(\cite{C:98}). On the other hand Vickers and Wilson (\cite{VW:00}) used
Colombeau generalized functions to show local 
well-posedness of the wave equation in conical space-times modelling a cosmic
string. This result has been generalized
to a class of locally bounded space-times in \cite{GMS:09}, also see
\cite{dude} for the static case and \cite{H:11} for an extension to non-scalar
equations.

This work is organized in the following way. In Section \ref{SCP} we recall the
classical theory
of normally hyperbolic operators on smooth Lorentzian manifolds, thereby essentially following a
recent book by B\"ar et.\ al.\ 
(\cite{BGP:07}). We extend their global existence and uniqueness result (Th.\
\ref{mainwave}) to the case of distributional data and
right hand side in Section \ref{dcp}. In Section \ref{gga} we
recall the necessary background from global
analysis based on Colombeau generalized functions (\cite{GKOS:01}). 
We devote Section \ref{local}
to presenting a variant of the local existence 
and uniqueness theorem for the wave equation in weakly singular space-times of
\cite{GMS:09}. Finally, in Section \ref{global}
we extend this result to a global theorem for weakly singular space-times which
allow for a suitable globally hyperbolic splitting of the metric.

\section{A review of wave equations on smooth\\ Lorentzian manifolds}\label{SCP}

In this section we present the solution theory for wave equations on smooth Lorentzian 
manifolds with smooth right hand side and data (cf.\ (\ref{CP}) below),
mainly based on \cite{BGP:07, W:09} (see also \cite{F:75}).
We will at once formulate the theory for normally hyperbolic
operators on sections of a vector bundle $E$ over some Lorentzian manifold $M$. 

We first fix some notations. Throughout we will assume $(M,g)$ to be a smooth, connected, time-oriented
Lorentz manifold and $E\to M$ a smooth vector bundle over $M$. By $\Gamma(M,E)$ or $\Gamma(E)$
we denote the smooth sections of $E\to M$, $\D(M,E)$ is the space of compactly supported 
sections, and $\Gamma^k(E)$ denotes the spaces of sections of finite differentiability. 
Given a finite-dimensional vector space $W$, the space of $W$-valued distributions
in $E$, $\D'(M,E,W)$ comprises the continuous linear maps $\D(M,E^*)\to W$. For example, 
given $x\in M$, the delta-distribution $\delta_x$ is the $E_x^*$-valued distribution in $E$
given by $\delta_x: \D(M,E^*) \to E_x^*$, $\varphi\mapsto \vphi(x)$.

A differential operator $P:\Gamma(E)\to \Gamma(E)$ of second order 
is called {\em normally hyperbolic} if its principal
symbol satisfies $\sigma_P(\xi)|_x = -\langle \xi,\xi\rangle \cdot \mathrm{id}_{E_x}$ ($\xi\in T_x^*M$). If $E$ is
equipped with a linear connection $\nabla$ then an important example of a normally hyperbolic
operator is given by the connection d'Alembertian
$$
\Box^\nabla:
\Gamma(E) \stackrel{\nabla}{\longrightarrow} \Gamma(T^*M\otimes E) \stackrel{\nabla}{\longrightarrow}
\Gamma(T^*M\otimes T^*M\otimes E)
\stackrel{-\mathrm{tr}\otimes \mathrm{id}_E}{\longrightarrow} \Gamma(E),
$$
where $\mathrm{tr}: T^*M\otimes T^*M \to \R$, $\mathrm{tr}(\xi\otimes\eta)=\langle \xi, \eta
\rangle$ is the metric trace. By the Weitzenb\"ock formula, for any normally hyperbolic operator
on $\Gamma(E)$ there exists a unique connection $\nabla$ on $E$ and a unique $B\in 
\Gamma(\mathrm{Hom}(E,E))$ such that $P=\Box^\nabla + B$.
For any differential operator $D:\Gamma(E) \to \Gamma(F)$, its formal adjoint
$D^*:\Gamma(F^*) \to \Gamma(E^*)$ is uniquely characterized by
$$
\int \psi(D\vphi)\,dV = \int (D^*\psi)(\vphi)\,dV
$$
for all $\vphi\in \D(M,E)$ and all $\psi\in \D(M,F^*)$. Here $dV$ is the Lorentzian 
volume density on $M$.

For notions from causality in Lorentz manifolds our main references are 
\cite{ON:83,BEE:96}. In particular, for $p,\ q\in M$,
by $p<q$ (resp.\ $p<<q$) we mean that there exists a causal (resp.\ timelike) future directed
curve from $p$ to $q$. For $A\subseteq\Omega\subseteq M$ we write 
$$
I^\Omega_+(A) := \{q\in \Omega \mid \exists p\in A \text{ s.t. } q>> p \text{ in } \Omega\}
$$
for the relative chronological future of $A$ in $\Omega$, and analogous for $I^\Omega_-(A)$.
Also we write 
$$
J^\Omega_+(A) := \{q\in \Omega \mid \exists p\in A \text{ s.t. } q\geq p \text{ in } \Omega\},
$$
for the relative causal future and analogous for $J^\Omega_-(A)$. Here $q\geq p$ means either $p<q$ or $p=q$.
A domain $\Omega$ is called causal if its closure $\overline\Omega$ is contained in a geodesically convex domain
$\Omega'$ and if for all $p,q\in \overline\Omega$ the causal diamond $J^{M}_+(p)\cap
J^{M}_-(q)$ is a compact subset of $\Omega'$. The manifold
$M$ is said to satisfy the {\em causality condition} (CC) if there are no closed causal curves in $M$.
It satisfies the {\em strong causality condition} (SCC) if for each $p\in M$ and
each neighborhood $U$ of $p$ there exists a neighborhood $V\subseteq U$ such that  
no causal curve that starts and ends in $V$ can leave $U$ (i.e., there are no `almost
closed' causal curves in $M$). 

A {\em Cauchy hypersurface} is a subset $S$ of $M$ that is intersected by each inextendible timelike
curve exactly once. $M$ is called {\em globally hyperbolic} if it satisfies SCC and for all
$p,q\in M$ the causal diamond $J^{\Omega'}_+(p)\cap
J^{\Omega'}_-(q)$ is compact. The following is a very useful characterization of global
hyperbolicity, due in its final form to \cite{BS:05}. 
\begin{thm} \label{BS} For any time-oriented Lorentzian manifold $M$, the following are equivalent:
\begin{itemize}
\item[(i)] $M$ is globally hyperbolic.
\item[(ii)] There exists a Cauchy hypersurface in $M$.
\item[(iii)] $M$ is isometric to $(\R \times S, -\beta dt^2 + g_t)$, where $\beta=\beta(t,x)$ is
smooth and strictly positive, $t\mapsto g_t$ is a smooth family of Riemannian
metrics on $S$, and each $\{t\}\times S$ is a (smooth) spacelike Cauchy-hypersurface in $M$.   
\end{itemize}
\end{thm}

We remark that by a recent result (\cite{BS:07}) global hyperbolicity is also characterized by
the condition that results from replacing SCC by CC in the above definition. 

Turning now to the problem of solving the initial value problem for a normally hyperbolic
differential operator on $E\to M$, we first consider the case where $M = (V,\langle\ ,\,\rangle)$
is a Lorentz vector space of dimension $n$ (later on, the role of $V$ will be played by
a tangent space $T_xM$ to $M$). 
We denote by $\gamma: V\to \R$, $\gamma(X):= -\langle X,X\rangle$ the
quadratic form associated with $\langle\ ,\,\rangle$. The analytic centerpiece of the entire
construction that is to follow is provided by the so-called {\em Riesz distributions}:
\begin{defn} For $\alpha\in\C$ with $\mathrm{Re}(\alpha)>n$, let 
$$
R_{\pm}(\alpha)(X) := \left\{ 
\begin{array}{rl}
C(\alpha,n) \gamma(X)^{\frac{\alpha-n}{2}} & \text{ if } X\in J_{\pm}(0)\\
0 & \text{ else }
\end{array}
\right.
$$
where $C(\alpha,n):= \frac{2^{1-\alpha} \pi^{\frac{2-n}{2}}}{\Gamma(\frac{\alpha}{2})
\Gamma(\frac{\alpha-n}{2} + 1)}$.
$R_+(\alpha)$ ($R_-(\alpha)$) is called advanced (retarded) Riesz distribution on $V$.
\end{defn}
$R_\pm(\alpha)$ is continuous on $V$, and using the fact that $\Box R_\pm(\alpha+2) = R_\pm(\alpha)$
for $\mathrm{Re}(\alpha) > n+2$, $\alpha \mapsto R_\pm(\alpha)$ uniquely extends to a holomorphic
family of distributions on all of $\C$. For all $\alpha$, $\supp(R_\pm(\alpha))\subseteq
J_\pm(0)$ and $\singsupp(R_\pm(\alpha))$ is contained in the boundary $C_\pm(0)$ of
$J_\pm(0)$. Moreover, $R_\pm(0) = \delta_0$.

The next step in the construction is to transport the Riesz distributions onto the Lorentz
manifold $M$. To this end, let $\Omega$ be a normal neighborhood of $x\in M$ and define the
smooth function $\mu_x: \Omega \to \R$ by $dV=\mu_x \cdot (\exp_x)_*(dz)$, where 
$dz$ is the standard volume density on $T_xM$, and
$\exp_x$ is the exponential map at $x$. In normal coordinates around $x$, $\mu_x = 
\sqrt{|\det g_{ij}|}$. Now we set $R^\Omega_\pm(\alpha,x) := \mu_x \cdot (\exp_x)_* R_\pm(\alpha)$,
i.e.,
$$
\forall \vphi\in \D(\Omega,\C): \quad \langle R^\Omega_\pm(\alpha,x),\vphi\rangle
= \langle R_\pm(\alpha), (\mu_x\cdot \vphi)\circ \exp_x \rangle\,.
$$
$R^\Omega_\pm(\alpha,x)$ are called {\em advanced} ({\em retarded}) {\em Riesz distributions}
on $\Omega$. The analytical properties of $R_\pm(\alpha)$ carry over to the manifold
setting, albeit in slightly more involved form, e.g. (setting $\Gamma_x:=\gamma\circ\exp_x^{-1}$),
\begin{equation}\label{rxa}
\Box R^\Omega_\pm(\alpha+2,x) = \left(\frac{1}{2\alpha}
(\Box\Gamma_x -2n)+1 \right) R^\Omega_\pm(\alpha,x)\quad\text{for $\alpha\not=0$}.
\end{equation}
Moreover, $R^\Omega_\pm(0,x) = \delta_x$, $\supp(R^\Omega_\pm(\alpha,x))\subseteq
J^\Omega_\pm(x)$ and $\singsupp(R^\Omega_\pm(\alpha,x))\subseteq C^\Omega_\pm(x)$.

For any normally hyperbolic operator $P$ on $\Gamma(E)$, our aim is to construct a
fundamental solution in the following sense:
\begin{defn}
Let $P:\Gamma(E)\to \Gamma(E)$ be normally hyperbolic. A distribution $F\in \D'(M,E,E_x^*)$ 
such that $PF = \delta_x$ (i.e., $\langle F,P^*\vphi\rangle = \vphi(x)$ for all $\vphi\in 
\D(M,E^*)$) is called a fundamental solution of $P$ at $x\in M$. $F$ is called advanced
(retarded) if $\supp(F)\subseteq J^M_+(x)$ ($\subseteq J^M_-(x)$).
\end{defn}
For example, in a Lorentz vector space $V$ as above, $R_\pm(2)$ is an advanced (retarded)
fundamental solution of $\Box$ at $0$ since $\Box R_\pm(2) = R_\pm(0) = \delta_0$.
In the manifold setting matters are more complicated, as already indicated in (\ref{rxa}).
We first make the following formal ansatz for a fundamental solution on a normal neighborhood
$\Omega$ of $x$:
\begin{equation}\label{formal}
{\mathcal R}_\pm(x) := \sum_{k=0}^\infty V^k_x R^\Omega_\pm(2+2k,x)\,,
\end{equation}
where $V^k_x\in \Gamma(\Omega,E\otimes E_x^*)$, the so-called {\em Hadamard-coefficients},
are to be determined. By formal termwise differentiation one finds that in order for
${\mathcal R}_\pm(x)$ to be a fundamental solution, the $V^k_x$ have to satisfy the 
following transport equations:
$$
\nabla_{\grad \Gamma_x} V^k_x - (\frac{1}{2}\Box\Gamma_x -n + 2k)V^k_x = 2kPV^{k-1}_x
\quad (k\ge 0)
$$
with $V^0_x(x) = \mathrm{id}_{E_x}$. The Hadamard coefficients are therefore uniquely
determined as the solution to this problem. 

Next, introducing convergence-generating factors into (\ref{formal}) we obtain
an approximate fundamental solution $\tilde {\mathcal R}(x)$ in the following sense:
$$
\exists K_\pm \in \Gamma(\bar \Omega \times \bar \Omega,E^* \boxtimes E) \text{ s.t.\ }
P_{(2)} \tilde {\mathcal R}(x) = \delta_x + K_\pm(x,\,.\,)\,.
$$
Here, $E^* \boxtimes E$ denotes the exterior tensor product and $P_{(2)}$ indicates
that $P$ acts on the second variable. We now use the $K_\pm$ as integral kernels to define
for any continuous section $u$ of $E^*$ over $\bar\Omega$:
$$
({\mathcal K}_\pm u)(x) := \int_{\bar \Omega} K_\pm(x,y) u(y)\,dV(y)\,.
$$
For $\Omega$ a sufficiently small causal domain and any $k$, $\mathrm{id} + {\mathcal K}_\pm:
\Gamma^k(\bar\Omega,E^*) \to \Gamma^k(\bar\Omega,E^*)$ is an isomorphism with bounded inverse
given by the Neumann series
$$
(\mathrm{id} + {\mathcal K}_\pm)^{-1} = \sum_{j=0}^\infty (-{\mathcal K}_\pm)^j\,.
$$
Finally, for each $\vphi\in \D(\Omega,E^*)$, we set 
$F^\Omega_\pm(\,.\,)[\vphi] := (\mathrm{id} + {\mathcal K}_\pm)^{-1}(\tilde{\mathcal R}_\pm
(\,.\,)[\vphi]) \in \Gamma(E^*)$. Then
\begin{eqnarray*}
\Gamma(E^*) &\to& E_x^* \\
\vphi &\mapsto& F^\Omega_\pm(x)[\vphi]
\end{eqnarray*}
is an advanced (retarded) fundamental solution for $P$ at $x$. Indeed, 
\begin{eqnarray*}
P_{(2)}F^\Omega_\pm(\,.\,)[\vphi] &=& F^\Omega_\pm(\,.\,)[P^*\vphi] = 
(\mathrm{id} + {\mathcal K}_\pm)^{-1}(\tilde{\mathcal R}_\pm(\,.\,)[P^*\vphi])  \\
&=& (\mathrm{id} + {\mathcal K}_\pm)^{-1}(P_{(2)}\tilde{\mathcal R}_\pm(\,.\,)[\vphi])
= (\mathrm{id} + {\mathcal K}_\pm)^{-1}(\vphi+{\mathcal K}_\pm\vphi) = \vphi
\end{eqnarray*}

In addition, for each $\vphi\in \D(\Omega,E^*)$, the map $x'\mapsto F^\Omega_\pm(x')[\vphi]$
is in $\Gamma(\Omega,E^*)$.
It can be shown
that the approximate fundamental solution $\tilde {\mathcal R}_\pm(x)$ is in fact an
asymptotic expansion of the true fundamental solution $F^\Omega_\pm(\,.\,)(x)$ in 
a suitable sense. Altogether, we obtain the following result on the solution of the
inhomogeneous problem on small domains:
\begin{thm} For any $x\in M$ there exists a relatively compact causal neighborhood $\Omega$
such that the following holds: given $v\in\D(\Omega,E)$ and defining $u_\pm$ by
$$
\langle u_\pm,\vphi\rangle := \int_\Omega F^\Omega_\pm(x)[\vphi]\cdot v(x)\,dV(x)
\quad (\vphi\in \D(\Omega,E))
$$
we have: $u_\pm \in \Gamma(\Omega,E)$, $Pu_\pm = v$, and $\supp(u_\pm) \subseteq 
J_\pm^\Omega(\supp(v))$.
\end{thm}
Turning now to the global theory, the first step is to assure uniqueness of fundamental 
solutions. For this, we need certain restrictions on the causal structure of $M$:
\begin{thm} \label{locuni}
Suppose that $M$ satisfies the causality condition, that the relation $\le$
is closed on $M$, and that the time separation function $\tau: M\times M \to \bar \R$
(see \cite{ON:83}) is finite and continuous. Let $u\in \D'(M,E)$ be a solution of
$Pu=0$ with future or past compact support ($\supp(v)\cap J^M_\pm(p)$ compact for all $p$).
Then $u=0$. 
\end{thm}
As an immediate corollary we obtain that under the above assumptions for each $x\in M$
there is at most one fundamental solution for $P$ at $x$ with past (future) compact
support. The causality conditions in Th.\ \ref{locuni} are satisfied if $M$ is globally
hyperbolic. Under this assumption, by Th.\ \ref{BS} there exists a spacelike Cauchy
hypersurface $S$ in $M$. We denote by $\hat\xi$ the future directed timelike unit normal vector 
field on $S$ and consider the following Cauchy problem:
\begin{eqnarray} \label{CP}
Pu &=& f \quad \text{ on } M \nonumber\\
u &=& u_0 \quad \text{ on } S\\
\nabla_{\hat\xi} u &=& u_1 \quad \text{ on } S \nonumber
\end{eqnarray}
where $f\in \D(M,E)$, and $u_0,\, u_1 \in \D(S,E)$. One first notes that for vanishing $f$, $u_0$,
and $u_1$, this problem only has the trivial solution. The analytical core of this result is 
the observation that for each $\psi\in \Gamma(E^*)$ and each $v\in \Gamma(E)$, one has
\begin{equation}\label{diveq}
\psi\cdot (Pv) - (P^*\psi)\cdot v = \div (W)
\end{equation}
where the vector field $W$ is uniquely characterized by
$$
\langle W,X\rangle = (\nabla_X \psi)\cdot v - \psi\cdot (\nabla_X v) \qquad (X\in {\mathfrak X}(M))\,.
$$
This allows to control the solution of the homogeneous equation $Pu=0$ by the Cauchy data on any Cauchy 
hypersurface. To prove existence of solutions one uses the above local theory to obtain solutions for $f$, $u_0$ and $u_1$
supported in sufficiently small causal domains. More precisely, we call a causal domain $\Omega$ an RCCSV-domain (for {\em relatively compact causal with small volume}) if it
is relatively compact and so small that $\mathrm{vol}(\overline\Omega)\cdot
\|K_\pm\|_{{\mathcal C}^0(\overline\Omega\times \overline\Omega)} < 1$. Then we have: 
\begin{prop}\label{RCCSV}
Let $\Omega$ be an RCCSV-domain and suppose that
$f$, $u_0$ and $u_1$ are compactly supported in $\Omega$ (resp.\ $\Omega\cap S$). Then
the corresponding Cauchy problem in $\Omega$ is uniquely solvable.
\end{prop}
Combined with rather subtle causality arguments, this local result
finally leads to the following main theorem on existence and uniqueness of solutions to (\ref{CP}):
\begin{thm} \label{mainwave}
Let $E$ be a vector bundle over a globally hyperbolic Lorentz manifold $M$ with  a spacelike Cauchy hypersurface $S$, and
let $P$ be normally hyperbolic on $\Gamma(E)$. Then for each $f\in \D(M,E)$ and each $u_0$, $u_1 \in \D(S,E)$ there 
exists a unique solution $u\in \Gamma(E)$ satisfying (\ref{CP}). Moreover, $\supp(u) \subseteq
J^M(\supp(f)\cup \supp(u_0) \cup \supp(u_1))$  and $u$ depends continuously on $(f,u_0,u_1)$.
\end{thm}
This result can immediately be utilized to show existence and uniqueness of fundamental solutions:
\begin{thm} \label{fundsol}
Under the assumptions of Th.\  \ref{mainwave}, for each $x\in M$ there exists a unique fundamental solution
$F_+(x)$ ($F_-(x)$) for $P$ at $x$ with past (future) compact support. These fundamental solutions satisfy
\begin{itemize}
\item[(i)] $\supp(F_\pm(x)) \subseteq J_\pm^M(x)$.
\item[(ii)] $\forall \vphi \in \D(M,E^*)$, $x\mapsto F_\pm(x)[\vphi] \in \Gamma(E^*)$ and $P^*(F_\pm(\,.\,)[\vphi])
= \vphi$.
\end{itemize}
\end{thm}
The corresponding fundamental kernels are called Green operators:
\begin{thm} \label{Greenop}
Under the assumptions of Th.\  \ref{mainwave}, there exist unique Green operators $G_\pm: \D(M,E) \to \Gamma(M,E)$
satisfying
\begin{itemize}
\item[(i)] $P\circ G_\pm =\mathrm{id}_{\D(M,E)}$
\item[(ii)] $G_\pm\circ P |_{\D(M,E)} = \mathrm{id}_{\D(M,E)}$
\item[(iii)] $\forall \vphi \in \D(M,E):$ $\supp(G_\pm \vphi)\subseteq J^M_\pm(\supp(\vphi))$
\end{itemize}
In fact, $(G_\pm\vphi)(x) = F_\mp(x)[\vphi]$. Moreover, denoting by $G^*_\pm$ the Green operators for $P^*$,
$$
\int_M(G^*_\pm\vphi)\cdot \psi \, dV = \int_M \vphi\cdot (G_\mp\psi)\,dV \qquad(\vphi\in \D(M,E^*),\, \psi\in \D(M,E))\,.
$$
\end{thm}

\section{Wave equations on smooth Lorentzian \\ manifolds: the case of distributional data}\label{dcp}

In this section we will be concerned with the global Cauchy problem in the case where the metric is
still smooth but the data and right hand side are distributional. To keep the presentation simple we will
restrict our attention to the wave operator.

Based on Th.\ \ref{BS}, we will assume that $M=\R\times S$ and the Lorentz metric
on $M$ is of the form $\lambda = -\beta dt^2 + g_t$. We denote by $\Box$ the 
d'Alembertian w.r.t.\ $\lambda$. The Cauchy problem we are considering can then
be written as
\begin{eqnarray} \label{CPD}
\Box u &=& f \quad \text{ on } M \nonumber\\
u(0,\,.\,) &=& u_0 \\
\nabla_{\hat\xi} u(0,\,.\,) &=& u_1 \nonumber
\end{eqnarray}
In the present setting, $\hat\xi = \frac{1}{\sqrt{\beta}}\partial_t$.
To make this initial value problem 
meaningful in the distributional setting we suppose that $f$ is smooth in
the $t$-variable. More precisely, we assume that 
$f\in {\mathcal C}^\infty(\R,\E'(S))\cap \E'(\R\times S)$.
It then follows from non-characteristic regularity (\cite[Th.\ 8.3.1]{Hoermander:V1})
that any $u\in \D'(M)$ with $\Box u=f$ has $\pm(1,0)$ not in the wave front set $\WF(u)|_{(t,x)}$ for all $(t,x)$. 
 Thus by \cite[23.65.5]{D:93}, $u\in {\mathcal C}^\infty(\R,\D'(S))$, so the initial value problem
(\ref{CPD}) indeed makes sense for $u_0,\, u_1\in \E'(S)$. 

To our knowledge, (\ref{CPD}) has not been treated in full generality in the literature so 
far for $f$, $u_0$, $u_1$ as specified above. We therefore supply the necessary arguments.
\begin{lem} \label{u1}
(Uniqueness) There is at most one solution $u\in {\mathcal C}^\infty(\R,\D'(S))$
of $($\ref{CPD}$)$.
\end{lem} 
\begin{proof} We first note that $\Box$ is strictly hyperbolic with respect to the level sets of 
the map $T:\R\times S \to \R$, $(t,x)\mapsto t$ (in the sense of \cite[Def.\ 23.2.3]{Hoermander:V3}). 
In fact, for the principal symbol $\sigma$ of $\Box$ we have
$$
\sigma|_{(t,x)}(\tau,\xi) = \frac{1}{\beta(t,x)}\tau^2 - h_t^{-1}(x)(\xi,\xi)\,.
$$
Thus fixing $(t,x)\in M$ and $\xi\in T_x^*S$, $\xi\not=0$, the polynomial $p(z) :=
\sigma|_{(t,x)}(z,\xi)$ has the distinct real zeros $\pm\sqrt{\beta(t,x)h_t^{-1}(\xi,\xi)}$.

Suppose now that $u$ and $\tilde u$ are solutions of (\ref{CPD}) and set $w:=u-\tilde u$. 
By strict hyperbolicity it follows that there exists some neighborhood $V$ of $S$ such that
$w|_V = 0$ (see \cite[23.72.8]{D:93}). Furthermore, by \cite[Th.\ 23.2.9]{Hoermander:V3}, 
$\WF(w) \subseteq \Char(\Box)$ and is invariant under the Hamiltonian flow
of $\sigma$. Now $\Char(\Box)$ consists entirely of lightlike directions. Thus the projection
of any (maximal) bicharacteristic onto $M$ is an inextendible null-geodesic, hence intersects
the Cauchy surface $S$. Since the wavefront set of $w$ is empty in $V$ and is transported
along the bicharacteristics it therefore must be empty everywhere on $M$. Hence $w\in 
{\mathcal C}^\infty(M)$, and $w=0$ follows from the uniqueness part of Th.\ \ref{mainwave}.
\end{proof}
Turning now to the problem of existence, it clearly suffices to treat the following
special cases of (\ref{CPD}): on the one hand, the homogeneous problem ($f=0$), which
we denote by (CP1), and on the other hand the inhomogeneous 
problem with vanishing initial data, called (CP2). 

Turning first to (CP2), from Th.\ \ref{Greenop} it is straightforward to 
conclude that the Green operators $G_\pm:\D(M)\to {\mathcal C}^\infty(M)$ continuously
extend to operators from $\E'(M) \to \D'(M)$ as transposed operators of $G^*_\mp$ 
(see \cite[Th.\ 4.3.10]{W:09}). Thus given $f\in {\mathcal C}^\infty(\R,\E'(S))\cap \E'(\R\times S)$
we may set $w:=G_+ f$ to obtain $\Box w = f$. 

Our task is thereby reduced to proving solvability
of (CP1) since adding the solution of (CP1) with $u_0= -w(0,\,.\,)$ and 
$u_1= -\nabla_{\hat\xi} w(0,\,.\,)$
we obtain the desired solution of (CP2) (note that since $\supp(G_+f) \subseteq J^M_+(\supp f)$
by \cite[(4.3.20)]{W:09} both $u_0$ and $u_1$ are in $\E'(S)$ by \cite[Cor.\ A.5.4]{BGP:07}).

To obtain a solution of (CP1) we first observe the following consequence of (\ref{diveq}),
(cf.\ \cite[Th.\ 4.3.20]{W:09}): 
Denote by $G^*:=G^*_+ - G^*_-$ the {\em propagator} of the transposed operator $\Box^*$.
Then any smooth solution $u$ of the homogeneous equation $\Box u = 0$ satisfies
for all $\vphi\in \D(M)$
\begin{equation}\label{propagator}
\int_M \vphi \cdot u \, dV = \int_S (\nabla_{\hat\xi}G^*(\vphi))\cdot u_0
- G^*(\vphi)\cdot u_1 \,dA
\end{equation}  
where $u_0 = u|_S$, $u_1=\nabla_{\hat\xi}u|_S$, and $dA$ is the Riemannian surface
element of the spacelike surface $S$ (i.e., the Riemannian density w.r.t.\ $g_0$ in our case).  
For the distributional Cauchy problem (CP1), we take (\ref{propagator})
as a starting point and for given $u_0,\, u_1 \in \E'(S)$ define $L(u_0,u_1)\in \D'(M)$
by
$$
\langle L(u_0,u_1),\vphi \rangle := \langle u_0, (\nabla_{\hat\xi}G^*(\vphi))|_S \rangle
-\langle u_1,G^*(\vphi)|_S\rangle\,.
$$
In case $u_0$ and $u_1$ are the given Cauchy data $u:=L(u_0,u_1)$ will be the desired
solution to (CP1). We start establishing this fact by first deriving an explicit formula 
for $L(u_0,u_1)$ well-suited to the $(t,x)$-splitting. This naturally has to involve the 
Green operators. First observe that by Th.\ \ref{Greenop} (ii) we have $\Box L(u_0,u_1) = 0$, 
so the argument preceding Lemma \ref{u1} shows that in fact $L(u_0,u_1)\in {\mathcal C}^\infty(\R,\D'(S))$. 
Hence for $\vphi_0\in \D(\R)$, $\vphi_1\in \D(S)$ we may write
\begin{equation}\label{dings}
\lara{L(u_0,u_1),\vphi_0\otimes\vphi_1} = \int_\R \vphi_0(t) \lara{L(u_0,u_1)(t),
\vphi_1} \, dt.
\end{equation}

Denote by $F^*_\pm(s,x)$ the fundamental solutions of $\Box^*$ at $(s,x)$ according to 
Th.\ \ref{fundsol}. 
Setting $F^*_{s,x}:=F^*_-(s,x) - F^*_+(s,x)$, we have
$(G^*\vphi)(s,x) = \lara{F^*_{s,x},\vphi}$ for all $(s,x)\in M$ and all $\vphi\in \D(M)$.
Furthermore, 
$\Box F_{(s,x)} = \delta_{(s,x)} - \delta_{(s,x)} = 0$, so $F^*_{s,x}\in \Cinf(\R,\D'(S))$
for each $(s,x)\in M$. Thus for each $t\in \R$ and each $\psi\in \D(S)$ we obtain
\begin{equation}\label{Lt}
\lara{L(u_0,u_1)(t),\psi}
= \lara{u_0, \nabla_{\hat\xi}\lara{F^*_{s,x}(t), \psi} |_{s=0}} -
\lara{u_1, \lara{F^*_{0,x}(t), \psi }}\,.
\end{equation} 

We are now in the position to show that $u$ indeed attains the Cauchy data. To this end
choose sequences $u_0^{(j)}$, $u_1^{(j)}$ in $\D(S)$ that converge to $u_0$ resp.\
$u_1$ in $\E'(S)$. By Th.\ \ref{mainwave}, for each $j$ there exists a unique $u^{(j)}
\in \Cinf(M)$ such that
$$
\Box u^{(j)} = 0, \ u^{(j)}(0) = u_0^{(j)}, \ \nabla_{\hat\xi} u^{(j)}(0) = u_1^{(j)}.
$$
In addition, again by \cite[Th.\ 4.3.20]{W:09} we have $u^{(j)} = L(u_0^{(j)},u_1^{(j)})$.
Moreover, by \eqref{Lt} we obtain  for all $\psi\in \D(S)$ and all $t\in \R$
\begin{equation}\label{roconv}
\lara{u^{(j)}(t),\psi} = \lara{L(u_0^{(j)},u_1^{(j)})(t),\psi} \to
\lara{L(u_0,u_1)(t),\psi} = \lara{u(t),\psi}\,,
\end{equation}
i.e., for all $t$, $u(t) = \lim u^{(j)}(t)$ in $\D'(S)$. In particular, $u(0) = 
\lim u^{(j)}(0) = \lim u_0^{(j)} = u_0$, thereby verifying the first initial condition
for $u$. 

To show that $\nabla_{\hat\xi}u(0)=u_1$ we first observe that by \eqref{Lt} we have for all 
$\psi\in \D(S)$ and all $t\in \R$
\begin{eqnarray*}
 \lara{\partial_tu^{(j)}(t),\psi}&=&\lara{L(u_0^{(j)},u_1^{(j)})'(t),\psi}\\
  &\to&\lara{L(u_0,u_1)'(t),\psi}=\lara{\partial_tu(t),\psi}.
\end{eqnarray*}
Now since for all $t\in\R$ we have $\nabla_{\hat\xi}u(t)=1/\sqrt{\beta(t,.)}\,\partial_tu(t)$ we 
obtain
\begin{eqnarray*}
 \nabla_{\hat\xi}u(0)=\frac{1}{\sqrt{\beta(0,.)}}\,\partial_tu(0)&=&\frac{1}{\sqrt{\beta(0,.)}}\,\lim\partial_tu^{(j)}(0)\\
 &=&\lim\nabla_{\hat\xi}u^{(j)}(0)=\lim u_1^{(j)}=u_1,
\end{eqnarray*}
thereby also verifying the second initial condition.

Finally, we demonstrate that the support of the unique solution of \eqref{CPD} satisfies the same
inclusion relation as in the smooth case (Th.\ \ref{mainwave}). To see this, we first note
that for any $u\in \Cinf(M)$ such that $\supp(\Box u)$ is compact
we have the following generalization of (\ref{propagator}):
\begin{equation}\label{propagator2}
\int_M \vphi \cdot u \, dV = \int_M \Box u (G^*_+ + G^*_-)(\vphi)\,dV +
\int_S (\nabla_{\hat\xi}G^*(\vphi))\cdot u_0
- G^*(\vphi)\cdot u_1 \,dA
\end{equation}
(this follows by adapting the proof of \cite[Lemma 3.2.2]{BGP:07}). Now suppose that
$u$ is the unique solution of (\ref{CPD}) and pick a sequence $u_m\in \D(M)$ such
that $u_m(t) \to u(t)$ for each $t\in \R$. For any $\vphi\in \D(M)$ and each $t\in \R$,
$\supp((G^*_+ + G^*_-)\vphi(t,\,.\,))$ is compact, so
\begin{eqnarray*}
\langle \Box u_m, (G^*_+ + G^*_-)\vphi\rangle &=&
\int_\R\langle \Box u_m(t), (G^*_+ + G^*_-)\vphi(t,\,.\,)\rangle\,dt \\
&\to& \int_\R\langle \Box u(t), (G^*_+ + G^*_-)\vphi(t,\,.\,)\rangle\,dt
= \langle \Box u, (G^*_+ + G^*_-)\vphi\rangle
\end{eqnarray*}
Thus applying (\ref{propagator2}) to each $u_m$ and letting $m\to \infty$ we obtain
$$
\langle u,\vphi\rangle = \lara{f,(G^*_+ + G^*_-)\vphi} 
+ \langle u_0, (\nabla_{\hat\xi}G^*(\vphi))|_S \rangle
-\langle u_1,G^*(\vphi)|_S\rangle\,.
$$
From this and Th.\ \ref{Greenop} (iii), the claimed support properties of $u$ follow.
Summing up, we have proved:
\begin{thm}
Given $f\in {\mathcal C}^\infty(\R,\E'(S))\cap \E'(\R\times S)$ and $u_0$, $u_1\in \E'(S)$
there exists a unique solution $u\in \Cinf(\R,\D'(S))$ of the Cauchy problem (\ref{CPD}).
Moreover, $\supp(u) \subseteq J^M(\supp(f)\cup \supp(u_0) \cup \supp(u_1))$.
\end{thm}

\section{Generalized global analysis}\label{gga}

Colombeau algebras of generalized functions (\cite{Colombeau:84,Colombeau:85}) 
are differential algebras which contain the vector space of distributions and display 
maximal consistency with classical analysis in the light of the Schwartz impossibility 
result (\cite{Schwartz:54}). Here we review global analysis based on the special
Colombeau algebra $\G(M)$, for further details see \cite{DD:91,KS:02,KS:02b} and
\cite[Sec.\ 3.2]{GKOS:01}. 
The basic idea of its construction is regularization of distributions by nets of smooth 
functions and the use of asymptotic estimates in terms of a regularization parameter.

Let $M$ be a smooth, second countable Hausdorff manifold.
Set $I=(0,1]$ and denote by $\E(M)$ the subset of $\Cinf(M)^I$ consisting of all nets
depending smoothly on the parameter $\eps\in I$. The algebra 
of generalized functions on $M$ (\cite{DD:91}) is defined as the quotient $\G(M) := \EM(M)/\NN(M)$ of 
moderate modulo negligible elements of $\E(M)$, where the respective notions are
defined by the following asymptotic estimates. Here ${\mathcal P}$ denotes the space of all
linear differential operators on $M$.
\[
\EM(M) :=\{ (u_\eps)_\eps\in\E(M):\, \forall K\comp M\
\forall P\in{\mathcal P}\ \exists N:\ \sup\limits_{p\in K}|Pu_\eps(p)|=O(\eps^{-N}) \}\\
\]
\[\NN(M)  :=\{ (u_\eps)_\eps\in\E(M):\ \forall K\comp M\
\forall P\in{\mathcal P}\ \forall m:\ \sup\limits_{p\in K}|Pu_\eps(p)|=O(\eps^{m}) \}.
\]
Elements of $\G(M)$ are denoted by $u = [(u_\eps)_\eps] =
(u_\eps)_\eps + \NN(M)$. With componentwise operations, $\G(M)$ is a
fine sheaf of differential algebras where the derivations are Lie
derivatives with respect to smooth vector fields defined by
$L_Xu :=[(L_X u_\eps)_\eps]$, also denoted by $X(u)$. 


There exist embeddings $\iota$ of $\D'(M)$ into $\G(M)$ that are sheaf homomorphisms
and render $\Cinf(M)$ a subalgebra of $\G(M)$. Another, more coarse way
of relating generalized functions in $\G(M)$ to distributions is based on 
the notion of association:
$u\in \G(M)$ is called associated with $v\in \G(M)$, $u\approx v$, if $u_\eps - v_\eps \to 0$
in $\D'(M)$. A distribution $w\in \D'(M)$ is called the distributional
shadow of $u$ if $u\approx \iota(w)$.

The ring of constants in $\G(M)$ is the space $\tilde \C$ of generalized
numbers, which form the natural space of point values
of Colombeau generalized functions. These, in turn, are uniquely
characterized by their values on so-called (compactly supported) 
generalized points (equivalence classes of bounded nets $(p_\eps)_\eps$
of points, where $(p_\eps)_\eps \sim (q_\eps)_\eps$ if $d(p_\eps,q_\eps)
=O(\eps^m)$ for each $m$).

An element $u \in \G(M)$ is called globally bounded, if there exists a representative $(u_\eps)_\eps$ and a $C >0$ such that $|u_\eps(x)| \leq C$ for all $x \in M$ and all $\eps \in I$.

A similar construction is in fact possible for any locally convex space
$F$: an analogous quotient construction in terms of the seminorms
on $F$ allows to assign a $\tilde\C$-module $\G_F$ to $F$ in a natural way
(\cite{G:05}). A particularly important special case is obtained
for $F=\Gamma(M,E)$, the space of smooth sections of a vector bundle $E\to M$
(again with representatives that are supposed to depend smoothly on $\eps$). 
The resulting space $\GaG(M,E):=\G_F$ then is a
$\G(M)$-module, called the space of generalized sections  of the vector bundle
$E$. A convenient algebraic description is as follows:  
\begin{equation}
\label{tensorp} \GaG(M,E) = \G(M) \otimes_{\Cinf(M)} \Ga(M,E)= L_{\Cinf(M)}(\Ga(M,E^*),\G(M)).
\end{equation}
$\GaG$ is a fine sheaf of finitely generated and projective $\G$-modules. 
For the
special case of generalized tensor fields of rank $r,s$ we use the notation $\G^r_s(M)$. We have:
\[
 \G^r_s(M)\cong L_{\G(M)}(\G^1_0(M)^s,\G^0_1(M)^r;\G(M)).
\]
Observe that this allows the insertion of generalized vector fields and one-forms into
generalized tensors, a point of view which is essential when dealing with generalized
metrics in the following sense.

\begin{defn}
An element $g$ of $\G^0_2(M)$ is called a generalized pseudo-Riemannian metric
if it is symmetric ($g(\xi,\eta) = g(\eta,\xi)$ $\forall \xi,\, \eta\in \mathfrak{X}(M)$),
its determinant $\det g$ is invertible in $\G$ (equivalently, for each compact 
subset $K$ of a chart in $M$ there exists some $m$ such that $|\det (g_\eps)_{ij}(p)| > \eps^m$
for $\eps$ small and all $p\in K$), and it possesses a well-defined 
index $\nu$ (the index of $g_\eps$ equals $\nu$ for $\eps$ small).
\end{defn}
Based on this definition, many notions from (pseudo-)Riemannian geometry can be
extended to the generalized setting (cf.\ \cite{KS:02b}). In particular, any 
generalized metric induces an isomorphism between generalized vector fields and
one-forms, and there is a unique Levi-Civita connection corresponding to $g$.
This provides a convenient framework for non-smooth pseudo-Riemannian geometry and
the analysis of highly singular space-times in general relativity (see e.g.\ \cite[Ch.\ 5]{GKOS:01},
\cite{SV:06}).

For the purposes of the present work we will also need the notion of a time-dependent
generalized metric. To this end, let $S$ be an $n$-dimensional smooth manifold and let
$\text{pr}_2 \col \R\times S \to S$ denote the projection onto $S$. 
Also, let $\text{pr}_2^*(T^0_2 S)$ be
the corresponding pullback-bundle. 
\begin{defn}\label{tdepgmetric}
  An element $ {h} \in \Gamma_\G(\text{pr}_2^*(T^0_2 S))$ is called $t$-dependent generalized 
  pseudo-Riemannian metric if $h_t$ possesses a well-defined index and if
\begin{itemize}
\item[(i)] (Symmetry) $ {h}_t(\xi,\eta) =  {h}_t(\eta,\xi)$ in $\G(\R\times S)$ for all $\xi$, $\eta\in \mathfrak{X}(M)$. 
\item[(ii)] (Non-degeneracy) $\det ( {h}_t)$ is {\em strictly nonzero} in the following sense: 
for any $K\subset\subset \R$, and any $L$ compact in some chart neighborhood on $S$ there exists 
some $m$ such that $|\det (( {h}_\eps)_t)_{ij}(x)|>\eps^m$ for $(t,x)\in K\times L$ and $\eps$ small.
\end{itemize}
\end{defn}
We conclude this section by the following globalization lemma.
\begin{lem}\label{globallem} Let $u: I\times M \to N$ be a smooth map
and let (P) be a property attributable to values $u(\eps,p)$ 
that is stable with respect to decreasing $K$ and $\eps$ in the following sense:
if $u(\eps,p)$ satisfies (P) for all  $p\in K\subset\subset M$ and
all $\eps$ less than some $\eps_K>0$ then for any compact set $K'\subseteq K$
and any $\eps_{K'} \le \eps_K$, $u$ satisfies (P) on $K'$ for all
$\eps\le \eps_{K'}$. Then there exists a smooth map $\tilde u: I\times M \to N$
such that (P) holds for all $\tilde u(\eps,p)$ ($\eps\in I$, $p\in M$) and
for each $K\subset\subset M$ there exists some $\eps_K\in I$ such that
$\tilde u(\eps,p) = u(\eps,p)$ for all $(\eps,p)\in (0,\eps_K] \times K$.
\end{lem} 
\begin{proof}
Let $(K_l)_l$ be a compact exhaustion of $M$ with $K_l\subseteq K_{l+1}^\circ$ for all $l$
and choose a smooth function
$\eta: M \to \R$ with $0<\eta(x)\le \eps_{K_l}$ for all $x\in K_l\setminus K_{l-1}^\circ$ ($K_0:=\emptyset$) 
(cf.\ e.g.\ \cite[Lem.\ 2.7.3]{GKOS:01}). Moreover, let $\nu: \R^+_0 \to [0,1]$ be a smooth function satisfying 
$\nu(t)\le t$ for all $t$ and
$$
\nu(t) = \left\{
\begin{array}{ll}
t & 0 \le t \le \frac{1}{2} \\
1 & t \ge \frac{3}{2}                     
\end{array}
\right.
$$
For
$(\eps,x)\in I\times M$ let $\mu(\eps,x) := \eta(x) \nu\!\left(\frac{\eps}{\eta(x)}\right)$
and set $\tilde u(\eps,x) := u(\mu(\eps,x),x)$. Then $\tilde u$ has all the required 
properties. 
\end{proof}

\begin{rem} \label{globalrem} 
Lemma \ref{globallem} allows to globalize properties of Colombeau-type
generalized functions provided that representatives depend smoothly on $\eps$, as is assumed
throughout this work. Indeed, if in the above situation it is additionally 
assumed that $u$ is a representative of a Colombeau generalized function
then by the very nature of the defining asymptotic estimates,  $\tilde u$
is itself moderate and in fact constitutes a representative of the same
generalized function possessing the required property (P) globally on $M$. 
In particular, any (time-dependent) generalized metric $g$ possesses a
representative $(g_\eps)_\eps$ such that each $g_\eps$ is a smooth
(time-dependent) metric globally on $M$, a fact that will repeatedly be used in what is
to follow. 
\end{rem}

\section{Wave equations of non-smooth metrics I: The local theory}\label{local}

In this section we present a local existence and uniqueness result (closely related to the one in 
\cite{GMS:09}) for the homogeneous wave equation of a class of generalized Lorentz metrics $g$,
which will be extended to a global result in Section \ref{global}.
We start by introducing this class of generalized weakly singular metrics.

\subsection{Weakly singular Lorentzian metrics}

Let $g$ be a generalized Lorentzian metric on $M$. From now on we call the pair
$(M,g)$  a generalized space-time. To formulate
asymptotic conditions on representatives $({g_\eps})_\eps$ of $g$ 
let $ {m}$ be a background Riemannian metric on $M$ and denote by $\|\;\|_{{m}}$ the
norm induced on the fibers of the respective tensor bundle. 
To begin with we impose the following condition:

\begin{enumerate}
\label{settingmetric}
\item[(A)]
\label{new-setting1} For all compact sets $K$, for all orders of
derivative $k \in \mathbb N_0$ and all $k$-tuples of smooth vector fields
${\eta}_1, \dots, {\eta}_k$
and for any representative $( g_{\eps} )_\eps$ we have:
\begin{eqnarray}\label{conditionA}
\sup_K \left\Vert \mathscr{L}_{{\eta}_1} \dots
 \mathscr{L}_{{\eta}_k} g_{\eps}
 \right\Vert_{{m}} &=& O(\eps^{-k})\nonumber \quad (\eps\rightarrow 0)\\
\sup_K \left\Vert \mathscr{L}_{{\eta}_1}\dots\mathscr{L}_{{\eta}_k}
 g_{\eps}^{-1} \right\Vert_{{m}} &=&O(\eps^{-k})\quad(\eps\rightarrow 0).
\end{eqnarray}
\end{enumerate}

A generalized metric with property (A) will be called 
a \emph{weakly singular metric}. Note that here we use a somewhat different terminology
as compared to \cite{GMS:09}. 

Now with a view to formulating the local Cauchy problem of the wave operator
for such a metric 
we consider a local foliation of $M$ given by the level sets of some non-singular 
function $t\in\Cinf(U)$, where $U\subseteq M$ is open and relatively compact. 

To exclude trivial cases we require the level sets $\Sigma_{\tau}= \{ q \in U : t(q) = \tau \}$ 
to be space-like with respect to all $g_{\eps}$. In fact, we suppose a uniform variant of this condition 
which can also be viewed as a suitable generalization of the classical notion of time-orientability.
Moreover, we will need a uniform bound on the covariant derivative of the normal form which, in particular,
contains a condition on the second fundamental form of the level sets. More precisely, we demand:

\begin{enumerate}
\item[(B)]
\label{new-setting2} Each $p\in M$ possesses a neighborhood $U$ on which there exists a
 \emph{local time-function}, that is a smooth function $t$ with uniformly timelike differential $dt=:\bsig$, i.e.,
 \begin{equation}\label{b1}
  g_\eps^{-1}(\bsig,\bsig)\leq -C<0\quad\text{for some positive constant $C$}
 \end{equation}
 and one (hence any) representative $g_\eps$ and all small $\eps$. 
 In addition, we have for all $K\comp U$ that 
 \begin{equation}\label{b2}
  \sup_K \left\Vert \nabla^{\eps} {\sigma}
  \right\Vert_{{m}} = O(1)\quad(\eps\to 0),
 \end{equation}
where $\nabla$ is the covariant derivative of $g$.
\end{enumerate}

Let us denote the normal vector field to $\Sigma_\tau$ by
$\bxi$ (it is a \emph{generalized} vector field defined via its representative $\bxieps \in \mathfrak{X}(U)$, 
given by $\bsig = g_{\eps}(\bxieps, \cdot)$, i.e., $\bxieps=\grad_\eps t$). We 
observe that by pulling up the index, given (A), condition (B) is equivalent to 
$\sup_K \left\Vert \nabla^{\eps} {\xi}_\eps\right\Vert_{{m}} = O(1)$. Also for
vector fields $X, Y$ tangent to $\Sigma_\tau$, we obtain $\nabla^\eps \sigma (X,Y) =
Y(\sigma(X)) - \sigma(\nabla^\eps_Y X) = 0 - \sigma(\text{nor} \nabla^\eps_Y X)$, hence 
we obtain for the second fundamental form $II_\eps$ of the
hypersurfaces $\Sigma_\tau$
$$ 
   \left\Vert II_\eps  \right\Vert_{{m}} = O(1) 
    \qquad (\eps \to 0) \quad\text{uniformly on compact sets}.
$$

\begin{rem}
\label{ConditionsRemark}
Conditions~(A) and~(B) are given in terms of the $\eps$-asymptotics of
the generalized metric. There is, however, the following close
connection to the classical situation. Assume that we are given a space-time
metric that is locally bounded but not necessarily ${\mathcal C}^{1, 1}$ or of
Geroch-Traschen class (i.e., the largest class that allows
a consistent distributional treatment, see \cite{GT:87,LeFM:07} and
\cite{SV:09} for the relation with the present setting). We may then embed this
metric into the space of generalized metrics essentially by convolution with a
standard mollifier (for details again see \cite{SV:09}). From the explicit form of the embedding it is then
clear that condition (A) holds.

Condition~(B), in adapted coordinates, demands somewhat better asymptotics of
the time-derivatives of the spatial part of the metric as well as the spatial
derivative of the $(0,0)$-component.
This condition, in fact, is satisfied by several relevant examples as well.
In particular, we have:

\begin{itemize}
 \item Conical space-times fall into our class since estimates~(6) and~(7) in~\cite{VW:00} for the embedded 
  metric imply our condition~(A), while~(B) is immediate from the staticity of the metric.
 \item Impulsive pp-waves (in \lq\lq Rosen form\rq\rq) 
 \[
  ds^2= - du dv + \left( 1 + u_+ \right)^2 dx^2 + \left( 1 - u_+ \right)^2 dy^2
 \]
 as well as expanding impulsive waves with line element
 \[
  ds^2=2 du dv + 2 v^2 \left| dz + \frac{u_+}{2v} \overline{H} d\overline{z} \right|^2
 \]
 satisfy conditions (A) and (B). Here  $u_+ $ denotes the kink function and
 $H(z)$ is the Schwarzian derivative $H=h'''/(2h')-(3h''^2)/(4h'^2)$ of some analytic 
function $h(z)$ (which may be chosen
 arbitrarily). For details see \cite[Ch.\ 20]{GP:09}.  In both cases the metric is continuous and
 it will obey conditions~(A) and (B) when embedded with a standard mollifier, or -- more generally -- if we 
 use any regularization that converges locally uniformly to the
 original metric.
\end{itemize}
\end{rem}

\subsection{Local existence and uniqueness}

We start by formulating the local Cauchy problem for the wave operator on weakly singular
space-times. Let $p \in M$, and choose $U$ to be an open and relatively compact 
neighborhood of $p$ as in condition (B).
Denote the corresponding foliation by $\Sigma_{\tau}= \{ q \in U : t(q) = \tau \}$ ($\tau\in[-\ga,\ga]$)
and suppose $p\in\Sigma_0=:\Sigma$. 
In addition to the normal vector field $\bxi$ and the normal covector field
$\bsig$ whose asymptotics have already been discussed above we will need
their corresponding normalized versions $\bxihat = [(\bxiepshat)_\eps] =
[(\bxieps / V_{\eps})_{\eps}]$ and $\bsighat =
[(\bsigepshat)_{\eps}] = g( \bxihat, \cdot)$, where we have set 
$ V_{\eps}^2 = - g_{\eps}(\bxieps, \bxieps)$.

We are interested in the initial value problem 
\begin{eqnarray}
\Box u &=& 0 \quad\text{on $U$}\nonumber\\\label{weqofsetting} 
u &=& u_0\quad\text{on $\Sigma$}\\
\nonumber \nabla_{\widehat{{\xi}}} u &=& u_1\quad\text{on $\Sigma$},
\end{eqnarray}
where the initial data $u_0$, $u_1$ are 
supposed to be in $\mathcal{G}(\Sigma)$. Note that this, in particular,
includes the case of distributional initial data. We are
interested in finding a local solution $u \in \mathcal{G}$ on $U$
or an open subset thereof.

A general strategy to solve PDEs in generalized functions is the following.
First, solve the equation for fixed $\eps$ in the smooth setting and
form the net $(u_\eps)_\eps$ of smooth solutions. This will be a
candidate for a solution in $\mathcal{G}$, but particular care has
to be taken to guarantee that the $u_\eps$ share a common domain of
definition and depend smoothly on $\eps$. In fact, as has recently been shown in \cite{BK:11} it suffices to verify continuous dependence on $\eps$. 	 In the second step, one shows that the solution
candidate $(u_\eps)_\eps$ is a moderate net, hence obtaining
existence of a solution $[(u_\eps)_\eps]$ in $\mathcal{G}$.
Finally, to obtain uniqueness of solutions, one has to prove that
changing representatives of the data leads to a
solution that is still in the class $[(u_\eps)_\eps]$. Note that
this amounts to an additional stability of the equation with respect
to negligible perturbations of the initial data.

So, in the present situation we need a condition which provides us with the existence
of a solution candidate:

\begin{enumerate}
 \item [(C)] For each $p\in \Sigma$ there exists a neighborhood 
$V\subseteq U$ and a representative $(g_{\eps})_\eps$ of the metric $g$ on $V$ such 
that $V$ is, for each $\eps$, an RCCSV-neighborhood in $(M,g_{\eps})$ with $\Sigma\cap V$ a 
spacelike Cauchy hypersurface for $V$.
\end{enumerate}

Indeed Prop.\ \ref{RCCSV} now provides us with a solution candidate defined on $V$, that
is a net $(u_\eps)$ with $\Box_\eps u_\eps=f_\eps$ on $V$ for some negligible net $(f_\eps)$ 
and moreover $u_\eps|_{\Sigma\cap V}=u_{0\eps}$, $\nabla_{\hat\xi} u|_{\Sigma\cap V}=u_{1\eps}$ for 
some representatives $(u_{0\eps})$, $(u_{1\eps})$ of the data.
We note that continuous dependence of $u$ on $\eps$ follows readily from the construction steps detailed in Section \ref{SCP}.
Observe that the only part that exceeds the classical condition for existence and uniqueness of solutions 
is a certain uniformity in $\eps$. Heuristically this means that the light-cones of the metric $g_{\eps}$ 
neither vary to wildly with $\eps$ nor collapse as $\eps \rightarrow 0$. In terms of regularizations of 
classical metrics which are locally bounded but not necessarily of ${\mathcal C}^{1,1}$ or of 
Geroch-Traschen class, this condition will always be satisfied due to non-degeneracy of the classical metric. 

Now we may state the main result of this section.

\begin{thm}[Local existence and uniqueness of generalized solutions]
\label{mainthm} Let $(M, g)$ be a generalized space-time  
with a weakly singular metric and assume that conditions (B) and (C) 
hold
. Then, for each $p \in \Sigma$ there exists an open neighborhood $\Om$ such 
that for all compactly supported $u_0,u_1\in\G(\Sigma\cap\Om)$, the initial value problem~\eqref{weqofsetting} 
has a unique solution $u$ in $\mathcal{G}(\Om)$.
\end{thm}

The core of the proof of Th.\ \ref{mainthm} consists of higher order energy
estimates for the solution candidate whose existence is secured by condition (C).
The energy estimates, which rely on conditions (A) and (B), will be carried out in a geometric 
setting using $\eps$-dependent energy momentum tensors and $\eps$-dependent Sobolev norms. 
These notions have been introduced in \cite{VW:00} and a suitable generalization of this
method will be presented in the next section.

\subsection{Higher order energy estimates}\label{ee}

Let $U\supseteq V$ be the neighborhoods of $p$ given by conditions (B) and (C). 
The solution candidate $(u_\eps)_\eps$ is defined on 
$V$ and we are going to estimate $u_\eps$ on some suitable neighborhood of $p$.

\begin{figure}[ht!]
\begin{center}
\includegraphics[width=5in]{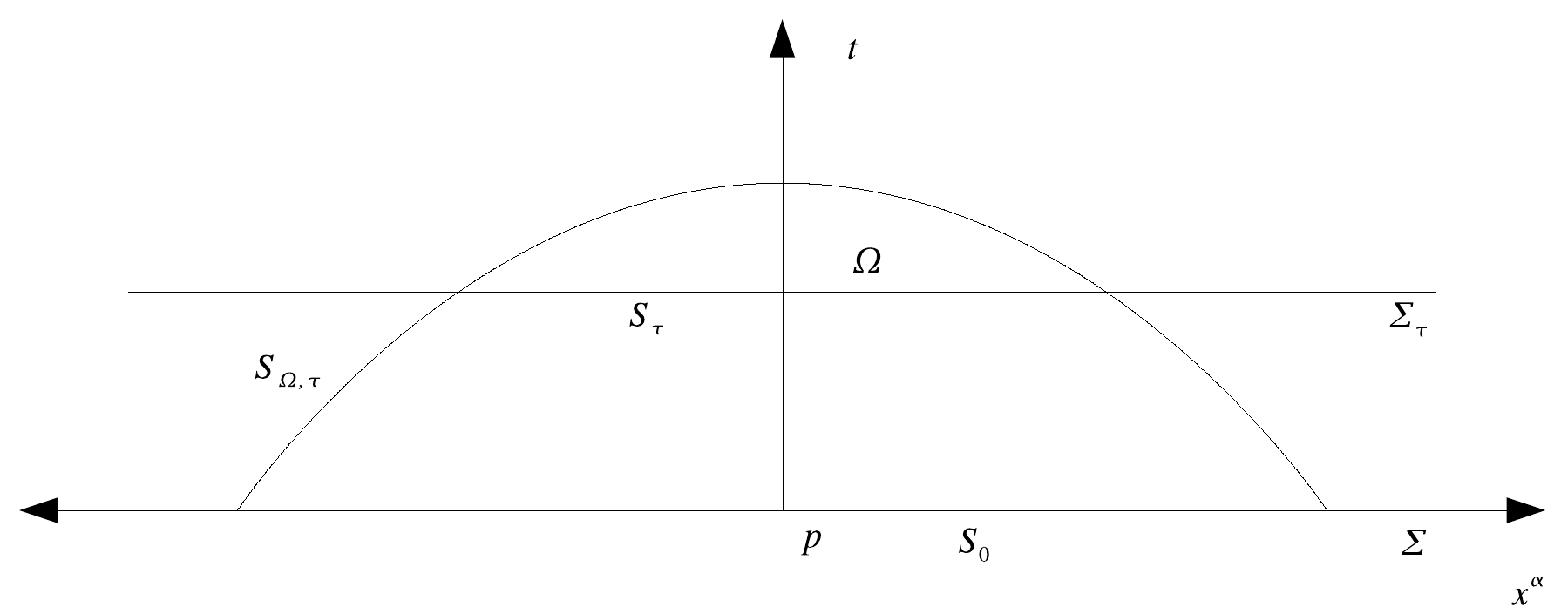}
\end{center}
\caption{Local foliation of space-time} \label{fig:1}
\end{figure}
We start by introducing some notation. Let $\Omega$ be an open neighborhood of
$p$ with the property that $\overline{\Omega} \subseteq V$, and such
that the boundary of the region $\Omega \cap \{ q \in U: t(q) \ge 0
\}$  and the boundary of the region $\Omega \cap \{ q \in U: t(q) \le 0
\}$ is space-like w.r.t.\ all $g_\eps$. Indeed, such a neighborhood exists
by condition (A). 
We now concentrate on the forward-in-time part of $\Om$ (i.e., the part where $t\ge 0$)
since the backward-in-time part can be dealt with analogously. To this end we set
$S_\tau := \Sigma_\tau \cap \Omega$ and denote by $\Omega_\tau$ the open
part of $\Omega$ between $\Sigma$ and $\Sigma_\tau$. We denote the
part of the boundary of $\Omega_{\tau}$ with $0 < t < \tau$ by
$S_{\Omega, \tau}$, so that $\partial \Omega_{\tau} = S_0 \cup
S_{\tau} \cup S_{\Omega, \tau}$ (see Figure~\ref{fig:1}).

From now on we will adopt abstract index notation for (generalized)
tensorial objects (see e.g. \cite{PR:87}). In particular,
representatives of the metric $g_{\eps}$ and its inverse
will be denoted by $g^{\eps}_{ab}$ and $g_{\eps}^{ab}$,
respectively. 
In addition, to simplify the notation for tensors we
are going to use capital letters to abbreviate tuples of indices,
i.e., we will write $T^I_J$ for $T^{p_1\dots p_r}_{q_1\dots q_s}$
with $|I|=r$, $|J|=s$. 

We now define a generalized Riemannian metric by
\[
{e} := g + 2 \bsighat \otimes
\bsighat,
\]
and use it in combination with the covariant derivative $\nabla$ of $g$
to define ``$\eps$-dependent'' Sobolev norms and energies on $U$. Observe that
by conditions (A) and (B) we have $\|e_\eps\|_m=O(1)$ and $\|\nabla_\eps e_\eps\|_m=O(1)$ on compact sets. We will also frequently need tensor products of ${e}$ and use the notation
$e_{IJ}=e_{p_1q_1}\dots e_{p_rq_r}$ with $|I|=r=|J|$.

\begin{defn}[Sobolev norms and energies]
Let $T^I_J$ be a smooth tensor field and $u$ a smooth function on $U$, $0 \leq \tau \leq \gamma$, and $k, j \in \N_0$.
\begin{enumerate}
\item We define the pointwise\ norm of $T^I_J$ by
$
\|T^I_J\|_{e_\eps}^2 := e^\eps_{KL} e_\eps^{IJ} T^K_I T^L_J
$\\
and the higher order pointwise norm of  $u$ by
$
|\nabla_\eps^{(j)}u|^2
:=||\nabla^\eps_{p_1}\dots\nabla^\eps_{p_j}u||^2_{e_\eps}.
$

\item On $\Omega_\tau$ and $S_\tau$ we define the Sobolev norms 
\begin{eqnarray*}
\SobODt{u}{k}&:=&
\left(\sum_{j=0}^k \int_{\Omega_\tau}|\nabla_\eps^{(j)}(u)|^2\mu^\eps\right)^{\frac{1}{2}},
\ \text{and}\\ 
\SobSDt{u}{k}&:=&\left(\sum_{j=0}^k \int_{S_\tau}|\nabla_\eps^{(j)}(u)|^2\mu_\tau^\eps\right)^{\frac{1}{2}}
\end{eqnarray*}
respectively. Here $\mu=[(\mu^\eps)_\eps]$ and $\mu_\tau=[(\mu_\tau^\eps)_\eps]$ denote the respective volume forms 
on $\Omega_\tau$ and $S_\tau$ derived from $g$.
Note that although in the second norm the integration is performed over
the three-dimensional manifold $S_\tau$ only, derivatives are not
confined to directions tangential to $S_\tau$.

\item On $\Omega$ we define the energy momentum tensors by ($k > 0$)
\begin{eqnarray*}
T^{ab, 0}_{\eps}(u) &:=&-\frac{1}{2}g^{ab}_\eps u^2,
\\
T^{ab, k}_{\eps}(u) &:=&\big( g^{ac}_\eps
g^{bd}_\eps-\frac{1}{2}g^{ab}_\eps g^{cd}_\eps \big)e^{p_1q_1}_\eps
\dots e^{p_{k-1}q_{k-1}}_\eps
\\&&\hskip 1cm \times
(\nabla_c^\eps \nabla_{p_1}^\eps \dots \nabla_{p_{k-1}}^\eps u)
(\nabla_d^\eps \nabla_{q_1}^\eps \dots \nabla_{q_{k-1}}^\eps u),
\end{eqnarray*}
\item Finally, the energy integrals are defined by
\begin{equation}
E^k_{\tau, \eps}(u):= \sum_{j=0}^k \int_{S_\tau} T^{ab, j}_\eps(u)
\xi_a {\xi}_b\, V_\eps^{-1}\mu_\tau^\eps. 
\label{HEintegrals}
\end{equation}
\end{enumerate}
\end{defn}

A straightforward calculation shows that the tensor fields $T^{ab, k}_{\eps}(u)$ satisfy 
the dominant energy condition hence an application of Stokes' theorem yields the basic energy estimate
(see e.g.\ \cite[Sec.\ 4.3]{HE:73})
\begin{equation}
\label{energyhierarchystokes} E^k_{\tau, \eps}(u) \leq E^k_{\tau=0,
\eps}(u) + \sum_{j=0}^k \int_{\Omega_\tau} \left( \xi_b
\nabla_a^\eps T^{ab, j}_\eps(u) + T^{ab, j}_\eps(u) \nabla_a^\eps
\xi_b \right)\mu_\eps.
\end{equation}

One key estimate in our approach is the equivalence of Sobolev norms and
energies. Indeed using condition (A) and the estimate \eqref{b1} one may 
derive (\cite[Lemma 4.1]{GMS:09}):

\begin{lem}[Energy integrals and Sobolev norms]
\label{lemma1}
There exist constants $C, C'$ such that for each $k\geq 0$ and all $\eps$ small 
\begin{equation}
\label{ineqEXSD} C(\SobSDt{u}{k})^2 \leq E^k_{\tau, \eps}(u) \leq
C'(\SobSDt{u}{k})^2.
\end{equation}
\end{lem}

With this tool at hand we may derive 
the core estimate allowing to prove existence and uniqueness of
solutions.

\begin{prop}
\label{energyinequality} Let $(u_{\eps})$ be a solution candidate
on $V$. Then, for each $k \ge 1$, there
exist positive constants $C_k', C_k'', C_k'''$ such that for all
$0\leq\tau\leq\gamma$ we have
\begin{eqnarray}
\label{energyinequalitylevelkformula} E^k_{\tau, \eps}(u_\eps)
&\leq& E^k_{0, \eps}(u_\eps) +C_k'(\SobODt{f_\eps}{k-1})^2
+C_k''\sum_{j=1}^{k-1}\frac{1}{\eps^{2(1+k-j)}}
\int_{0}^\tau E_{\zeta, \eps}^j(u_\eps)d\zeta\nonumber\\
&&+C_k''' \int_{0}^\tau E_{\zeta, \eps}^k(u_\eps) d\zeta.
\end{eqnarray}
\end{prop}

Before sketching the proof of this statement, we draw the essential conclusions
from it. Observe that the constant in front of the highest order
term on the r.h.s.\ does not depend on $\eps$, hence we obtain, by
an application of Gronwall's lemma:

\begin{cor}\label{energyestimate} Let
$(u_{\eps})$ be a solution candidate on $V$.
Then, for each $k \geq 1$, there exist positive constants $C_k', C_k'',
C_k'''$ such that for all $0 \leq \tau \leq \gamma$,
\begin{eqnarray}
\label{applicationenergygronwallformula}
 \lefteqn{E^k_{\tau,\eps}(u_\eps)} \\ \nonumber
 &&\leq \left(E^k_{0, \eps}(u_\eps)+C_k'
(\SobODt{f_\eps}{k-1})^2+C_k''\sum_{j=1}^{k-1}\frac{1}{\eps^{2(1+k-j)}}
\int\limits_{\zeta=0}^\tau E_{\zeta, \eps}^j(u_\eps) d\zeta\right)
e^{C_k'''\tau}
\end{eqnarray}
Consequently, if the initial energy $(E^k_{0, \, \eps}(u_\eps))_\eps$ is a moderate (resp.\
negligible) net of real numbers, and $(f_\eps)_\eps$ is negligible then 
\[
\sup_{0 \leq \tau \leq \gamma} (E^k_{\tau, \, \eps}(u_\eps))_\eps
\]
is moderate (resp.\ negligible).
\end{cor}

Now we \emph{sketch the proof of Proposition \ref{energyinequality}}: We have 
to estimate the right hand side of the basic energy estimate \eqref{energyhierarchystokes}.
Starting with the second term under the integral we use condition (B) to obtain 
\begin{equation}\label{modi1}
\left| T^{ab, j}_\eps(u_{\eps}) \nabla_a^\eps \xi_b \right| 
\le\Vert T^{ab, j}_\eps(u_{\eps}) \Vert_{{e}_{\eps}} \Vert
\nabla_{(a}^\eps \xi_{b)} \Vert_{{e}_{\eps}} 
\leq C\,\Vert T^{ab, j}_\eps(u_{\eps}) \Vert_{{e}_{\eps}}.
\end{equation}
Observe that by condition (A), the Riemannian metric ${e}$ is $O(1)$ hence $\|\ \|_m$ and $\|\ \|_{e_\eps}$ 
are equivalent norms. We now may estimate 
$\Vert T^{ab, j}_\eps(u_{\eps}) \Vert_{{e}_{\eps}}$ by the higher order pointwise norm of $u$. 
After integration this gives
\begin{equation}\label{term2}
\left| \sum_{j=0}^k \int_{\Omega_\tau} T^{ab, j}_\eps(u_{\eps})
 \nabla_a^\eps \xi_b \mu_\eps \right| 
\ \leq\ C\, (\SobODt{u_{\eps}}{k})^2.
\end{equation}

Now turning to the divergence term on the right hand side of \eqref{energyhierarchystokes}
we start with orders $k=0,1$. Using the wave equation we find
\[
\nabla_a^\eps T^{ab, 0}_\eps(u_\eps) = -u_\eps \nabla_\eps^b u_\eps\quad \text{and}\quad 
\nabla_a^\eps T^{ab, 1}_\eps(u_\eps) = f_\eps \nabla^b_\eps u_\eps,
\]
which after integration clearly can be estimated by the squares of $\SobODt{u_{\eps}}{1}$ and $\SobODt{f_\eps}{0}$. 
Inserting this and \eqref{term2} for $k=1$ into \eqref{energyhierarchystokes} we obtain
\[
E^1_{\tau, \eps}(u_\eps) 
\leq 
E^1_{0, \eps}(u_\eps) +C\,\left(\SobODt{f_\eps}{0}\right)^2 + C\left(\SobODt{u_\eps}{1}\right)^2.
\] 
Next we use Lemma \ref{lemma1} to estimate the Sobolev norm of $u_\eps$ in terms of its energy, i.e.,
\[
\left(\SobODt{u_\eps}{1}\right)^2 =
\int_{0}^\tau(\SobSDz{u_\eps}{1})^2 d\zeta \leq C
\int_{0}^\tau E^1_{\zeta, \eps}(u_\eps)d\zeta,
\]
which gives the claim for $k=1$ (with $C_1''=0$).

Finally, one has to estimate the divergence of the higher (i.e., $k>1$) order energy momentum tensors.
The general strategy is, of course, to rewrite terms containing the $(k+1)^{\text{st}}$ order derivative of $u_\eps$
using the wave equation. This necessitates interchanging the oder of covariant derivatives, which
introduces additional curvature terms. These can be estimated using condition (A). Observe, however, that there
also appear terms where the covariant derivatives falls on ${e}$. These terms of the form 
$\nabla^\eps_a e^{IJ}_\eps$ can be estimated thanks to condition (B). \hfill $\Box$
\medskip

We finally \emph{sketch the proof of Th.\ \ref{mainthm}}: We have already noted the 
existence of a solution candidate $(u_\eps)$ on $V$. To prove that $u_\eps$ 
is moderate on $\Omega_\gamma$ we start from moderateness of the data $u_0,u_1$. Inductively using the wave equation
this translates into moderateness of the initial energies $(E^k_{0, \eps}(u_\eps))_\eps$.
Now by Corollary~\ref{energyestimate} we obtain moderateness
of the energies $(E^k_{\tau,\eps}(u_\eps))_\eps$ for all $0\leq \tau\leq \gamma$.
Finally, we use the Sobolev embedding theorem (together with the fact the volume is $O(1)$
due to condition (A)) to estimate the sup-norm of $u_\eps$
in terms of the Sobolev norms, which in turn can be bounded by the energies $(E^k_{\tau,\eps}(u_\eps))_\eps$
due to Lemma \ref{lemma1} (for details see \cite[Lemma 6.2]{GMS:09}). So we see that
moderateness of the energies implies moderateness of $(u_\eps)$ and we have proved existence of solutions.

Uniqueness follows along the same lines
replacing moderateness by negligibility.\hfill $\Box$

\section{Wave equations of non-smooth metrics II: The global
theory}\label{global}

We extend the results of Th.\ \ref{mainthm} to establish existence and uniqueness
of global generalized solutions. As in the classical situation we have to impose additional global
conditions on the generalized Lorentzian metric to control causality properties of
space-time in the large. Thus, we begin by transferring the notion of global
hyperbolicity to the setting of generalized space-times by appealing to a
variant in terms of the metric splitting property stated in Th.\ \ref{BS},
(iii).

\begin{defn}\label{genglobhyp} Let $g$ be a generalized Lorentz metric on the
smooth $(n+1)$-dimensional manifold $M$. We say that $(M,g)$ allows
a \emph{globally hyperbolic metric splitting} if there exists a
$\Cinf$-diffeomorphism $\psi \col M \to \R \times S$, where $S$ is an
$n$-dimensional smooth manifold such that the following holds for the pushed
forward generalized Lorentz metric ${\la} := \psi_*g$ on $\R\times
S$: 
\begin{trivlist}
\item{(a)} There is a representative $(\la_\eps)_{\eps\in I}$ of $\la$ such that
every $\la_\eps$ is a Lorentz metric and each slice $\{t_0\}\times S$ with
arbitrary $t_0 \in\R$ is a (smooth, spacelike) Cauchy hypersurface for every
$\la_\eps$ ($\eps\in I$).

\item{(b)} We have the metric splitting of $\la$ in the form 
$$
   {\lambda} = - \be dt^2  + {h}, 
$$
where ${h} \in \Gamma_\G(\text{pr}_2^*(T^0_2 S))$ is a $t$-dependent 
generalized Riemannian metric (in the sense of Def.\  \ref{tdepgmetric}) and
$\be \in \G(\R\times S)$ is globally bounded and \emph{locally uniformly
positive}, i.e., for some (hence any) representative $(\beta_\eps)$ of $\be$ and
 for every
$K \subset\subset \R \times S$ we can find a constant $C>0$ such that
$\beta_\eps(x)\geq C$ holds for small $\eps > 0$ and $x\in K$.

\item{(c)} For every $T > 0$ there exists a representative $({h}_\eps)$
of ${h}$ and a smooth complete Riemannian metric ${\rho}$ on $S$
which  uniformly bounds $h$ from below in the following sense: for all $t \in
[-T,T]$, $x \in S$, $v \in T_x S$, and $\eps \in I$
$$
   ({h}_\eps)_t (v,v) \geq {\rho}(v,v).
$$   

\end{trivlist}
\end{defn}

\begin{rem} Observe that the basic splitting structure and the requirements on 
lower bounds for $\be$ and ${h}$ in the above definition display common
features with the notion of regularly sliced space-times 
(\cite[Ch.\ XII, Subsec.\ 11.4]{Choquet-Bruhat:09}) which provide sufficient
conditions for global hyperbolicity in the smooth case (\cite{CBC:02}). 
\end{rem} 

\begin{ex} To obtain simple non-trivial examples of generalized space-times
satisfying the conditions of Def.\  \ref{genglobhyp} we consider
Robertson-Walker space-times. First, we briefly recall the classical situation:
Let $(S,h_0)$ be a connected Riemannian manifold, $f \col \R \to\, ]0,\infty[$
be smooth, and put $\la_{(t,x)} = -dt^2 + f(t)^2 (h_0)_x$ for every $(t,x)\in
\R\times S$. Then  the Lorentzian metric $\la$ on $\R \times S$ is globally
hyperbolic if and only if $(S,h_0)$ is complete (cf.\ \cite[Lemma
A.5.14]{BGP:07}). Moreover, if this is the case, then every slice $\{t_0\}
\times S$ is a smooth, spacelike Cauchy hypersurface.

We generalize the Robertson-Walker space-time by allowing as warping function
any $f \in \G(\R)$ that is globally bounded and locally uniformly positive and replace
$h_0$ by a generalized Riemannian metric on $S$ which is bounded below by some
smooth complete Riemannian metric (to guarantee condition (c)). By Lemma
\ref{globallem} on globalization techniques we may  pick representatives
$(f_\eps)$ of $f$ and $(h_{0\eps})$ of $h_0$ such that the smooth function
$f_\eps$ is everywhere positive and $h_{0\eps}$ is a Riemannian metric on $S$ for
every $\eps \in I$. In addition, we assume that 
each $h_{0\eps}$ is complete. The generalized Lorentz metric $\la := - dt^2 + f^2 h_0$
on $\R\times S$ then trivially satisfies condition (b). Putting $\la_\eps :=
-dt^2 + f_\eps^2  h_{0\eps}$ ($\eps \in I$) we obtain a representative of $\la$.
By completeness of the Riemann metric $h_{0\eps}$, the Lorentz metric $\la_\eps$
is globally hyperbolic and every slice $\{t_0\} \times S$ is a Cauchy
hypersurface for every $\eps \in I$. Thus condition (a) in Def.\ 
\ref{genglobhyp} is also satisfied.
\end{ex}

From now on we consider only generalized space-times $(M,g)$ which possess a
globally hyperbolic metric splitting. To simplify notation we will henceforth
suppress the diffeomorphism providing the splitting and assume that $M =
\R\times S$ and $g = \la$ with $S$ and $\la$ as in the statement of Def.\ 
\ref{genglobhyp}. Thus, the generalized space-time is represented by a family of
globally hyperbolic space-times $(M,g_\eps)$ such that $S \isom \{0\}\times S$
is a Cauchy hypersurface for every $g_\eps$ ($\eps \in I$).

Therefore we are provided with a suitable Cauchy hypersurface for the
initial value problem for the wave operator $\Box$ corresponding to the
generalized space-time metric $g$ on $M$, i.e., the Cauchy problem
\begin{eqnarray} \label{CPG}
\Box u &=& 0 \quad \text{ on } M \nonumber\\
u &=& u_0 \quad \text{ on } S\\
\nabla_{\bxihat} u &=& u_1 \quad \text{ on } S. \nonumber
\end{eqnarray}
Here the unit normal vector field of $S$ is given by $\bxihat=\frac{1}{\sqrt{\beta}}\partial_t$ 
and the initial data $u_0$, $u_1$ are assumed to belong to $\G(S)$ and to have
compact supports, e.g., arising by embedding distributional data from $\E'(S)$.

Now the key strategy to establish a global version of Th.\ \ref{mainthm} on
existence and uniqueness of solutions to the Cauchy problem \eqref{CPG} is as
follows: From the classical existence and uniqueness result in Th.\
\ref{mainwave} for every $\eps \in I$ we obtain a solution
candidate defined on all of $M$, which again depends continuously on $\eps$. 
In this sense the global hyperbolic metric
splitting of $(M,g)$ replaces condition (C) used in the proof of Th.\
\ref{mainthm} to produce a solution candidate. Then we aim at showing
moderateness, thus existence of a generalized solution, as well as uniqueness by
employing energy estimates as
in the local constructions of Section \ref{local}.  Therefore it is appropriate
to suppose also condition (A), i.e., $g$ to be weakly singular. In the present
situation we may translate \eqref{conditionA} into corresponding asymptotic
conditions on $\beta$ and ${h}_t$. 
As for condition (B), we see that the existence of a suitable (in this case even
global) foliation is a consequence of the globally hyperbolic metric splitting.
Indeed we globally have $g_\eps^{-1}(dt,dt) = -1 /\beta_\eps \leq - C < 0$ 
for some positive $C$, which implies \eqref{b1}.  
The second asymptotic boundedness condition \eqref{b2} in (B), i.e.,
$\sup_K \left\Vert \nabla^{\eps} dt  \right\Vert_{{m}} = O(1)$,
now simply reads
\begin{equation}
   \left\Vert d \beta_\eps \right\Vert_{{m}} = O(1)   \qquad(\eps \to 0)
\end{equation}
uniformly on compact sets. As in the local setting this implies for the 
extrinsic curvature of the hypersurfaces $\{t\}\times S$
$$ 
   \left\Vert II_\eps  \right\Vert_{{m}} = O(1) 
    \qquad (\eps \to 0) \quad\text{uniformly on compact sets}.
$$

We may now state the main result of this section. 

\begin{thm}[Global existence and uniqueness of generalized solutions] Let
$(M,g)$ be a generalized space-time with a weakly singular metric admitting a
globally hyperbolic metric splitting and assume that condition \eqref{b2}
holds. 
Then the Cauchy problem \eqref{CPG} has a unique solution $u\in\G(M)$ for all
compactly supported $u_0, u_1\in\G(S)$.
\end{thm}

\emph{Sketch of proof:}
Let $\Box_\eps$ denote the wave operator corresponding to $g_\eps$.  Then
Th.\ \ref{mainwave} provides us with a global solution $u_\eps$ to $\Box_\eps
= 0$ with Cauchy data $u_{0\eps}$ and $u_{1\eps}$, thereby defining a solution
candidate.

To prove existence we have to establish moderateness of the net $(u_\eps)_\eps$.
Choose an exhaustive sequence of compact sets $K_j$ ($j\in\N$) in $S$. Then it
suffices to show moderateness of $(u_\eps)_\eps$ on $L_j := [-j,j] \times K_j$
for each $j\in \N$.

Fix $j\in\N$ and choose $\rho$ as in Def.\  \ref{genglobhyp}, (c) with
$T=j$. We cover $L_j$ by finitely many lens-shaped regions as in Figure \ref{fig:1} that are, in turn, contained in coordinate neighborhoods. By the explicit construction
given in \cite[Sec.\  3.3.4]{dude} (based on condition (A)) the
heights of these lenses are uniformly bounded below  on $L_j$ by some $\de > 0$
due to the properties of $\rho$ and $\beta$. 

Employing the energy estimates from Section 4.3 we thereby derive moderateness
estimates of $u_\eps$ on the strip $[0,\de/2] \times K_j$ from the moderateness
of $u_{0\eps}$ and $u_{1 \eps}$. We may iterate this procedure to cover $[0,j]
\times K_j$ in finitely many steps, and analogously for $[-j,0] \times K_j$. 

By the same token negligibility of $(u_{0\eps})$ and $(u_{1\eps})$ implies the 
corresponding property for $(u_{\eps})$, which proves uniqueness. \hfill$\Box$ 
\medskip

Finally, we remark that our methods also allow to treat the inhomogeneous equation
as well as the inclusion of lower order terms, which have to satisfy certain asymptotic 
bounds, see \cite{H:11}. 

Also, condition (A) is actually a little stronger than what is needed to
prove moderateness of the solution candidate resp.\ negligibility in case of negligible data. 
Indeed, again by a result of \cite{H:11}, it suffices to suppose (A) for $k=0$. On the other hand
we could use (A) to explicitly calculate $\eps$-power bounds of (derivatives) of the solution,
which encode additional regularity information of our generalized solutions.


\subsection*{Acknowledgment} G.H.\ was supported by an EPSRC Pathway to Impact Award.
We also acknowledge the support of FWF-projects  Y237, P20525, and P23714.


\begin{thebibliography}{10}

\bibitem{BGP:07}
C.\ B{\"a}r, N.~Ginoux, and F.~Pf{\"a}ffle.
\newblock {\em Wave equations on {L}orentzian manifolds and quantization}.
\newblock ESI Lectures in Mathematics and Physics. European Mathematical
  Society (EMS), Z\"urich, 2007.

\bibitem{BEE:96}
J.K. Beem, P.E. Ehrlich, and K.L. Easley.
\newblock {\em Global {L}orentzian geometry}.
\newblock Marcel Dekker Inc., New York, 1996.

\bibitem{BS:05}
A.N. Bernal and M.~S{\'a}nchez.
\newblock {Smoothness of time functions and the metric splitting of globally
  hyperbolic spacetimes.}
\newblock {\em Commun. Math. Phys.}, 257(1):43--50, 2005.

\bibitem{BS:07}
A.N. Bernal and M.~S{\'a}nchez.
\newblock Globally hyperbolic spacetimes can be defined as `causal' instead of
  `strongly causal'.
\newblock {\em Classical Quantum Gravity}, 24(3):745--749, 2007.

\bibitem{BK:11} A.\ Burtscher and M.\ Kunzinger, 	
	Algebras of generalized functions with smooth parameter 	
	dependence. {\em Proc.\ Edinburgh Math.\ Soc.}, to appear.

\bibitem{Choquet-Bruhat:09}
Y.~Choquet-Bruhat.
\newblock {\em General relativity and the {E}instein equations}.
\newblock Oxford Mathematical Monographs. Oxford University Press, Oxford,
  2009.

\bibitem{CBC:02}
Y.~Choquet-Bruhat and S.~Cotsakis.
\newblock Global hyperbolicity and completeness.
\newblock {\em J. Geom. Phys.}, 43(4):345--350, 2002.

\bibitem{C:96}
C.~J.~S. Clarke.
\newblock Singularities: boundaries or internal points?
\newblock In P.~S. Joshi and A.~K. Raychaudhuri, editors, {\em Singularities,
  Black Holes and Cosmic Censorship}, pages 24--32. IUCCA, Bombay, 1996.

\bibitem{C:98}
C.~J.~S. Clarke.
\newblock Generalized hyperbolicity in singular spacetimes.
\newblock {\em Classical Quantum Gravity}, 15(4):975--984, 1998.

\bibitem{Colombeau:84}
J.~F. Colombeau.
\newblock {\em New generalized functions and multiplication of distributions}.
\newblock North-Holland, Amsterdam, 1984.

\bibitem{Colombeau:85}
J.~F. Colombeau.
\newblock {\em Elementary introduction to new generalized functions}.
\newblock North-Holland, 1985.

\bibitem{DD:91}
J.~W. de~Roever and M.~Damsma.
\newblock Colombeau algebras on a {$C^\infty$}-manifold.
\newblock {\em Indag. Math. (N.S.)}, 2(3):341--358, 1991.

\bibitem{D:93}
J.~Dieudonn{\'e}.
\newblock {\em Treatise on analysis. {V}ol. {VIII}}, volume~10 of {\em Pure and
  Applied Mathematics}.
\newblock Academic Press Inc., Boston, MA, 1993.

\bibitem{F:75} F.G.\ Friedlander. 
	{\em The wave equation on a curved space-time.} 	
	Cambridge Monographs on Mathematical Physics, Cambridge University Press, 1975. 	


\bibitem{G:05}
C.\ Garetto.
\newblock Topological structures in {C}olombeau algebras: topological
  {$\tilde{\mathbb C}$}-modules and duality theory.
\newblock {\em Acta Appl. Math.}, 88(1):81--123, 2005.

\bibitem{GT:87}
R.\ Geroch and J.\ Traschen.
\newblock Strings and other distributional sources in general relativity.
\newblock {\em Phys. Rev. D (3)}, 36(4):1017--1031, 1987.

\bibitem{GMS:09}
J.~Grant, E.~Mayerhofer, and R.~Steinbauer.
\newblock The wave equation on singular space-times.
\newblock {\em Commun. Math. Phys.}, 285(2):399--420, 2009.

\bibitem{GP:09}
J.~B. Griffiths and J.\ Podolsk{\'y}.
\newblock {\em Exact space-times in {E}instein's general relativity}.
\newblock Cambridge Monographs on Mathematical Physics. Cambridge University
  Press, Cambridge, 2009.

\bibitem{GKOS:01}
M.~Grosser, M.~Kunzinger, M.~Oberguggenberger, and R.~Steinbauer.
\newblock {\em Geometric theory of generalized functions}.
\newblock Kluwer, Dordrecht, 2001.

\bibitem{H:11}
C.~Hanel.
\newblock Wave-type equations of low regularity.
\newblock {\em Appl.\ Anal.}, to appear, 2011.

\bibitem{HE:73}
S.~W. Hawking and G.F.R. Ellis.
\newblock {\em {The large scale structure of space-time.}}
\newblock {Cambridge Monographs of Mathematical Physics. Vol. I. London:
  Cambridge University Press. XI, 391 p.}, 1973.

\bibitem{Hoermander:V3}
L.~H{\"o}rmander.
\newblock {\em The analysis of linear partial differential operators}, volume
  III.
\newblock Springer-Verlag, 1985.
\newblock Second printing 1994.

\bibitem{Hoermander:V1}
L.~H{\"o}rmander.
\newblock {\em The analysis of linear partial differential operators},
  volume~I.
\newblock Springer-Verlag, second edition, 1990.

\bibitem{KS:02}
M.~Kunzinger and R.~Steinbauer.
\newblock Foundations of a nonlinear distributional geometry.
\newblock {\em Acta Appl. Math.}, 71:179--206, 2002.

\bibitem{KS:02b}
M.~Kunzinger and R.~Steinbauer.
\newblock Generalized {pseudo-}{R}iemannian geometry.
\newblock {\em Trans. Amer. Math. Soc.}, 354(10):4179--4199, 2002.

\bibitem{LeFM:07}
P.~G. Le Floch and C.\ Mardare.
\newblock Definition and stability of {L}orentzian manifolds with
  distributional curvature.
\newblock {\em Port. Math. (N.S.)}, 64(4):535--573, 2007.

\bibitem{dude}
E.~Mayerhofer.
\newblock {\em The wave equation on static singular space-times}.
\newblock PhD thesis, University of Vienna, 2006.
\newblock http://arxiv.org/abs/0802.1616.

\bibitem{ON:83}
B.~O'Neill.
\newblock {\em Semi-{R}iemannian geometry}, volume 103 of {\em Pure and Applied
  Mathematics}.
\newblock Academic Press, New York, 1983.

\bibitem{PR:87}
R.\ Penrose and W.\ Rindler.
\newblock {\em Spinors and space-time. {V}ol.\ 1}.
\newblock Cambridge Monographs on Mathematical Physics. Cambridge University
  Press, Cambridge, 1987.
\newblock Two-spinor calculus and relativistic fields.

\bibitem{Schwartz:54}
L.~Schwartz.
\newblock Sur l'impossibilit{\'{e}} de la multiplication des distributions.
\newblock {\em C. R. Acad. Sci. Paris}, 239:847--848, 1954.

\bibitem{SV:06}
R.~Steinbauer and J.~A. Vickers.
\newblock The use of generalized functions and distributions in general
  relativity.
\newblock {\em Classical Quantum Gravity}, 23(10):R91--R114, 2006.

\bibitem{SV:09}
R.~Steinbauer and J.~A. Vickers.
\newblock On the {G}eroch-{T}raschen class of metrics.
\newblock {\em Classical Quantum Gravity}, 26(6):065001, 19, 2009.

\bibitem{VW:00}
J.~A. Vickers and J.~P. Wilson.
\newblock Generalized hyperbolicity in conical spacetimes.
\newblock {\em Class.~Quantum.~Grav.}, {\bf 17}:1333--1360, 2000.

\bibitem{W:09}
S.\ Waldmann.
\newblock Geometric wave equations. Lecture Notes, University of Freiburg.
\newblock \\ {\tt
  http://omnibus.uni-freiburg.de/$\sim$sw12/Lectures/Wellen0809/main.pdf},
  2009.

\end{thebibliography}
\end{document}